\documentclass[12pt]{amsart}
\usepackage[OT4]{fontenc}
\usepackage{graphics,graphicx}
\usepackage{amsmath}
\usepackage{amssymb,amsfonts}
\usepackage {amsthm}
\usepackage{xcolor}

\usepackage{enumerate}
\usepackage{hyperref}
\usepackage[T2A]{fontenc}
\usepackage[cp1251]{inputenc}
\usepackage[all]{xy}

\newtheorem{theorem}{Theorem}
\newtheorem{proposition}{Proposition}
\newtheorem{lemma}{Lemma}

\newtheorem*{proposition*}{Proposition}
\newtheorem{definition}{Definition}
\newtheorem{example}{Example}
\newtheorem{remark}{Remark}

\newcommand\myfootnote[1]{
\renewcommand{\thefootnote}{}
\footnotetext{#1}
\def\thefootnote{\@arabic\c@footnote}
}

\makeatletter
\renewcommand{\subsection}{\@startsection{subsection}{2}{0mm}{-\baselineskip}{-5pt}{\it \bf}}
\makeatother

\newcommand{\bigslant}[2]{{\raisebox{.2em}{$#1$}\left/\raisebox{-.2em}{$#2$}\right.}}

\title{Hamming spaces and locally matrix algebras}
\author{Oksana Bezushchak and Bogdana Oliynyk}

\thanks{The second author was partially supported by the grant for scientific researchers of the ``Povir u sebe'' Ukranian Foundation.}

\begin{document}
	
	\maketitle
	
	\address{Faculty of Mechanics and Mathematics,
		Taras Shevchenko National University of Kyiv,
		Volodymyrska, 60, Kyiv 01033, Ukraine \\ Department of Mathematics, National University
		of Kyiv-Mohyla Academy, Skovorody St. 2, Kyiv,
		04070, Ukraine                     }
	
	\email{bezusch@univ.kiev.ua, oliynyk@ukma.edu.ua}


\subjclass{2010}{Mathematics Subject Classification: 06E25, 15A30, 16S50}




\begin{abstract}
We introduce an abstract definition of a Hamming space that generalizes standard Hamming spaces  $( \mathbb{Z}/ 2 \mathbb{Z})^n $. We classify countable locally standard Hamming spaces and show that each of them can be realized as the Boolean algebra of idempotents of a Cartan subalgebra of a locally matrix algebra.
\end{abstract}




\section*{Introduction}

The Hamming metric was introduced in information theory as the most common tool to measure the difference between binary strings of the same length. The (standard){ \it Hamming space} $H_n$ is the set  of all $n$-tuples $ x^n=(x_1,\dots,x_n), $
$ x_i \in \{0,1\} ,$ $ 1 \leq i \leq n ,$ with   the distance $ d_{H_n},$ that is defined between two $n$-tuples as the number of coordinates
where they differ.

There are different generalizations of finite metric Hamming space to infinite case, that are constructed as inductive limits of finite Hamming spaces \cite{Cameron}, \cite{OlSusch}, \cite{Ol}.

 In this paper we define an (abstract) Hamming space and introduce the operation of tensor product. We call a Hamming space locally standard (see the definition below) if locally it looks as $H_{n}.$
In the first part of the paper we prove that every countable locally standard Hamming space is isomorphic to an infinite tensor product of standard Hamming spaces $H_{n}.$ In the second part of the paper we realize countable locally standard Hamming spaces as Cartan subalgebras of locally matrix algebras (see \cite{BezOl}, \cite{BezOl_2}, \cite{Koethe}, \cite{Kurosh}) and discuss conjugacy of Cartan subalgebras.

\section{Steinitz numbers  }

A  {\it Steinitz} number \cite{ST}  is a infinite formal
product of the form
$$ \prod_{p\in \mathbb{P}} p^{r_p}, $$
where $ \mathbb{P} $ is the set of all primes, $ r_p \in  \mathbb{N} \cup \{0,\infty\} $ for all $p\in \mathbb{P}$.
We can define the product of two Steinitz numbers by the rule:
$$ \prod_{p\in \mathbb{P}} p^{r_p} \cdot  \prod_{p\in \mathbb{P}} p^{k_p}= \prod_{p\in \mathbb{P}} p^{r_p+k_p}, \  r_p, k_p \in  \mathbb{N} \cup \{0,\infty\},  $$
where we assume, that
$$r_p+k_p=\begin{cases}
r_p+k_p, & \text{if  $r_p < \infty$ and $k_p < \infty$, } \\
\infty, & \text{in other cases}
\end{cases}.$$

By symbol $ \mathbb{SN} $ we denote the set of all Steinitz numbers. Obviously, the set of all positive integers $ \mathbb{N} $ is a  subset of the set of all Steinitz numbers $ \mathbb{SN} $. The elements of the set $ \mathbb{SN} \setminus
\mathbb{N} $ are called {\it infinite Steinitz } numbers.

\section{Hamming spaces }

Recall that a Boolean algebra is a commutative algebra over the field
 $ \bigslant{ \mathbb{Z}}{2 \mathbb{Z}} $ satisfying the identity $x^2=x$.

 \begin{definition}
 	\label{Def1}
 	By a (unital) Hamming space $(H,r)$ (see \cite{Cameron}, \cite{OlSusch}, \cite{Ol}) we mean a Boolean algebra $H$ with $1$ and a rang function $r \ : H \to [0,1] $ such, that
 	 \begin{enumerate}
 		\item[$(1)$]	\label{as1}
 		$r(a)=0$ if and only if $a=0$; 	
 		\item[$(2)$]\label{as2}
 		$r(a)=1$ if and only if $a=1$;
 		\item[$(3)$]\label{as3}
 		if $a,b \in H$ and $ab=0$, then $r(a+b)=r(a)+r(b)$.
 	\end{enumerate}

 \end{definition}

\begin{remark}
Note, that if  $(H,r)$ is a Hamming space then the function $$d_H(a,b)= r(a - b), \qquad a, b \in H,$$ makes the Hamming space $(H,r)$ a metric space.
\end{remark}

\begin{example}
The Boolean algebra $H_n=  \big(\bigslant{ \mathbb{Z}}{2 \mathbb{Z}}\big)^n$ with the rang function   $$r_{H_n}(x_1,\ldots,x_n)= \frac{1}{n}(x_1 +\ldots+x_n)$$ for all $x_1, \ldots,x_n \in \{0,1\}$ satisfy the assumptions $(1),$ $(2),$ $(3)$ of Definition $\ref{Def1}.$ We call the  Hamming space $(H_n,r_{H_n})$ standard. For all $a,b \in H_n$ the corresponding distance $d_{H_n}(a,b)$ equals  the number of coordinates where $a$ and $b$ differ.
\end{example}

Let $\{0,1\}^{\mathbb{N}}$ be the set of all (right-) infinite $(0,1)$-sequences. Clearly, $\{0,1\}^{\mathbb{N}}$ is a Boolean algebra under coordinate-wise addition (modulo 2) and multiplication.

\begin{example}	
An infinite sequence $\mathbf{a}=(a_1,a_2,\ldots) \in \{0,1\}^{\mathbb{N}}$
is said to be {\it periodic} if there exists a natural number $k$
such that the equality $a_{i}=a_{i+k}$ holds for all $i \in
\mathbb{N}$. In this case the number $k$ is called a {\it period} of the sequence
$\mathbf{a}$.

Let $u$ be a Steinitz number. A periodic sequence $\mathbf{a}$  is called {\it $u$-periodic}  if  its minimal
period is a divisor of $u$.

Let $\mathcal{H}(u)$ be the set of all $u$-periodic sequences. Clearly $\mathcal{H}(u)$ is a Boolean subalgebra of $\{0,1\}^{\mathbb{N}}$.  The rang function $$r_{\mathcal{H}(u)}(a_1,a_2,\ldots)=\frac{1}{k}(a_1 +\ldots+a_k),$$
where $k$ is a period of the sequence $(a_1,a_2,\ldots) ,$ makes $(\mathcal{H}(u),r_{\mathcal{H}(u)})$ a Hamming space.
\end{example}

\begin{example}
	\label{Bez}
For a sequence  $\mathbf{a}=(a_1,a_2,\ldots) \in \{0,1\}^{\mathbb{N}}$ define its pseudorang function
$$\Tilde{r}(\mathbf{a})= \lim_{n \to \infty} \sup \frac{1}{n}(a_1+\ldots+a_n).$$ Then $I=\{\mathbf{a} \in \{0,1\}^{\mathbb{N}} \ | \ \Tilde{r}(\mathbf{a})=0\}$ is an ideal of the Boolean algebra $\{0,1\}^{\mathbb{N}}$.

Consider the Boolean algebra $B=\bigslant{\{0,1\}^{\mathbb{N}}}{I}$ and the rang function $$r_{B}(a+I)=\Tilde{r}(\mathbf{a}), \quad \mathbf{a} \in \{0,1\}^{\mathbb{N}}.  $$ The Hamming space $(B,r_B)$ is called the Besicovitch space or the Besicovitch--Hamming space (see \cite{BlanchFormKurk}, \cite{Vershik1}).

\end{example}

\begin{definition}
	By a (non-unital) Hamming space $(H,r)$ (see \cite{Cameron}, \cite{Olijnyk}) we mean a  Boolean algebra $H$  with a rang function $r \ : H \to [0,\infty) $ such that
	\begin{enumerate}
		\item[$(1)$]
		$r(a)=0$ if and only if $a=0$; 	
		\item[$(2)$]
		if $a,b \in H$ and $ab=0$, then $r(a+b)=r(a)+r(b)$.
	\end{enumerate}	
\end{definition}

Similarly to the unital case,  we can define a metric $$d_H(a,b)= r(a - b), \qquad a, b \in H,$$ that  makes  $(H,r)$  a metric space.

\begin{example}
	Let $X$ be an infinite set and let $H$ be the non unital Boolean algebra of finite subsets of $X$, including the empty one. The rang function $r(a)=\# a$, $a \in H$, makes $(H,r)$ a non unital Hamming space (see \cite{Cameron}, \cite{Olijnyk}).
\end{example}

\begin{remark}
	For an arbitrary Hamming space $(H,r)$ and a nonzero element $h \in H$ consider the ideal $hH$ and the rang function $r_h \ : \ hH \to [0,1]$, $$r_h(a)=\frac{r(a)}{r(h)} , \qquad a \in hH.$$ Clearly, $(hH, r_h)$ is a unital Hamming space with the identity element $h$.
\end{remark}

\section{Tensor product of Hamming spaces}

The purpose of the following proposition is to define tensor product of Hamming spaces.

\begin{proposition}
	Let $(H_1,r_1)$, $(H_2,r_2)$ be Hamming spaces. Then there exists a unique rang function $r$ on $H=H_1\otimes_{\mathbb{Z}} H_2$ such that $r(a \otimes b)=r_1(a) r_2(b)$ for arbitrary elements $a \in H_1$, $b \in H_2$.
\end{proposition}

\begin{proof}
	Let $S(H)$ be the set of all nonempty finite subsets  of $H \setminus \{0\}$. Let
	\begin{multline*}
	E(H)=\{A \in S(H) \mid A=\{a_1, \ldots, a_r \}, \ a_i \ne 0, \ a_i a_j =0 \text{ for } i\ne j,\\ 1 \le i,j \le r \}.
	\end{multline*}
	
	We say that a set $X \in S(H)$ is covered by a set $E \in E(H)$ if for arbitrary elements $x \in X$, $e \in E$ we have $xe=e$ or $0$ and $$x=\sum_{e \in E,  \, xe=e}e.$$
	
	Let $X=\{a_1, \ldots, a_r\} \in S(H).$ Let the set $E$ consists of all nonzero products $b_1 \ldots b_r$, where $b_i=a_i$ or $1-a_i$. Then $E \in E(H) $ and $E$ covers $X$.
	
	For an arbitrary element $x \in H_1\otimes_{\mathbb{Z}} H_2$ there exist $E_1=\{e_1, \ldots, e_n\} \in E(H_1)$ and $E_2=\{f_1, \ldots, f_m\} \in E(H_2)$ such that $x \in (\text{Span } E_1) \otimes (\text{Span } E_2)$, $$x= \sum _{1\le i \le n, \, 1\le j \le m } \alpha_{ij} e_i \otimes f_j, \qquad \alpha_{ij}=0 \text{ or } 1.$$
	In this case we say that the element $x$ is covered by subsets $E_1$ and $E_2.$
	
	Define $$r_{E_1, E_2}(x)=\sum _{1\le i \le n, \, 1\le j \le m } \alpha_{ij} \ r_1(e_i) r_2(f_j).$$
	
	Let $E_1' \in E(H_1)$ and $E_2' \in E(H_2)$ such that $E_1$ and $E_2$ are covered by $E_1'$ and $E_2'$ respectively. Then
	$$x=  \sum_{1\le i \le n, \, 1\le j \le m } \alpha_{ij} \ e_p' \otimes f_q' ,$$ where the summation is done over all $e_p' \in E_1'$, $e_q' \in E_2'$ such that $e_ie_p'=e_p'$, $f_jf_q'=f_q'$. Hence
	$$r_{E_1', E_2'}(x)=\sum _{1\le i \le n, \, 1\le j \le m } \alpha_{ij} r_1(e_p') r_2(f_q'). $$ But $$r_1(e_i)= \sum_{e_i e_p^{'}=e_p^{'}} r_1(e_p^{'}) \ \ \text{ and } \ \ r_2(f_j)= \sum_{f_j f_q^{'}=f_q^{'}} r_2(f_q^{'}).$$
	This implies that $r_{E_1, E_2}(x)=r_{E_1^{\,'}, E_2^{\,'}}(x).$
	
	We claim that the function $r(x)=r_{E_1, E_2}(x)$ is well defined. Let $E_1, E_1^{\,'}\in E(H_1);$ $E_2, E_2^{\,'}\in E(H_2) $ and $$x \in (\text{Span }(E_1)\otimes \text{Span }(E_2)) \cap (\text{Span }(E_1^{\,'})\otimes \text{Span }(E_2^{\,'})).$$
	There exists $E_1^{\,''}\in E(H_1),$ $E_2^{\,''}\in E(H_2)$ such that $E_1$ and $E_1^{\,'}$ are both covered by $E_1^{\,''};$ $E_2$ and $E_2^{\,'}$ are both covered by $E_2^{\,''}.$  Then $$ r_{E_1, E_2}(x)=r_{E_1^{\,''}, E_2^{\,''}}(x)=r_{E_1^{\,'}, E_2^{\,'}}(x) .$$ 
	
	Define $r(x)=r_{E_1, E_2}(x).$ Let us show that  $r(x)$ is a rang function.
	
	Let $E_1=\{ e_1,\ldots, e_n \} \in E(H_1),$ $E_2=\{ f_1,\ldots, f_m \} \in E(H_2),$ $$x=\sum_{1\le i \le n, \, 1\le j \le m }   \alpha_{ij} \ e_i \otimes f_j, \ \alpha_{ij}=0 \text{ or } 1.$$ Clearly,
	$$r(x)=\sum_{1\le i \le n, \, 1\le j \le m } \alpha_{ij} \ r_1(e_i) r_2(f_j)\geq 0.$$ If $x\neq 0$ then there exist indices $i,j$ such that $\alpha_{ij}\neq 0.$ This implies $r(x) \geq r_1(e_i) \, r_2(f_j)> 0.$
	
	Also
$$	r(x)=\sum_{1\le i \le n, \, 1\le j \le m } \alpha_{ij} r_1(e_i) r_2(f_j)\leq \sum_{1\le i \le n, \, 1\le j \le m }  r_1(e_i) r_2(f_j)= $$ $$ =r_1(e_1+\cdots +e_n) \, r_2(f_1+\cdots + f_m)\leq 1.$$ The equality is achieved only when $e_1+\cdots +e_n=1,$ $f_1+\cdots + f_m=1$ and all $\alpha_{ij}=1,$ in which case $x=1$ in $H.$ Finally, let $x,y \in H,$ $xy=0.$ Let $$E_1=\{ e_1,\ldots, e_n \} \in E(H_1), \ E_2=\{ f_1,\ldots, f_m \} \in E(H_2)$$ cover  both $x$ and $y.$ Let  $$x=\sum_{1\le i \le n, \, 1\le j \le m }  \alpha_{ij} e_i \otimes  f_j, \ y=\sum_{1\le i \le n, \, 1\le j \le m }  \beta_{ij} e_i \otimes  f_j, \text{ then} $$ $$x+y=\sum_{1\le i \le n, \, 1\le j \le m }  (\alpha_{ij}+\beta_{ij}) e_i \otimes  f_j. $$ From $xy=0$ it follows that $\alpha_{ij} \beta_{ij}=0$ for all $i,j$ and therefore $\alpha_{ij}+\beta_{ij}=0$ or $1.$ Hence $$r(x+y) = \sum_{1\le i \le n, \, 1\le j \le m }  (\alpha_{ij}+\beta_{ij}) r_1(e_i) r_2 (f_j)=r(x)+r(y).$$ This completes the proof of the Proposition.\end{proof}

It is easy to see that $H_n \otimes H_m \cong H_{n m} .$

\begin{lemma}
\label{lemma_1}
Let $H_n \subset H_s.$ Then there is a subspace   $1\in H^{\,'}\subset H_s$  such that $$H_s = H_n H^{\,'} \cong H_n \otimes H^{\,'}, \  H^{\,'}\cong H_{s/n}.$$
\end{lemma}
\begin{proof}
The Boolean algebra $H_n$ contains $n$ orthogonal idempotents $e_1,$ $\ldots,$ $e_n,$ $1= \sum_{i=1}^{n} e_i,$ each of them has rang $ 1/n.$ Each element $e_i$ in $H_s$ has rang $r/s,$ $0\leq r \leq s.$ Hence $$\frac{1}{n} = \frac{r}{s}, $$ so $s=nr,$  which implies that $n|s.$  

An arbitrary element of $H_s$ of rang $r/s$ is a sum of $r$ orthogonal idempotents, each of rang $1/s.$ Let $$e_i = \sum_{j=1}^{s/n}  e_{i  j}, \ i=1, \ldots, n, \ \ e_j^{'} = \sum_{i=1}^{n}  e_{i  j}, \ \  j=1, \ldots, s/n.$$ The subalgebra $H^{\,'}$ of $H_s$ generated by $e_j^{'},$  $1\leq j \leq s/n,$  has the claimed properties. \end{proof}

\begin{definition}	\label{loc stand}  We say that a Hamming space $(H,r)$ is locally standard if an arbitrary finite collection of elements $a_1,$ $\ldots,$ $a_n \in H$ is contained in a subspace $H^{\,'}\subset H$ that is isomorphic to $H_m$ for some $m\geq 1.$
\end{definition}

\begin{example}
$u$-periodic Hamming space $\mathcal{H}(u)$ is locally standard for an arbitrary Steinitz number $u$. Indeed, for any finite collection of elements $a_1, \ldots, a_n \in \mathcal{H}(u)$  there exists a positive integer $m$ such that all sequences  $a_1, \ldots, a_n$ are periodic with period $m$ and $m | u$. Then  $a_1,\ldots, a_n $ are contained in a subspace $\mathcal{H}(u)$ that is isomorphic to $H_m$.
\end{example}

\begin{example} Note that the Besicovitch space $(X_B, r_B)$ is not locally standard because  we can construct  $x = (x_1, x_2, \ldots) \in X_B$ such that $ r_B(x)$ is irrational.
	
	Let $\{0,1\}^{\mathbb{N}}_p$ be the subset of $\{0,1\}^{\mathbb{N}}$ that consists of periodic sequences.  Clearly, $$\{0,1\}^{\mathbb{N}}_p \cap I= \{0\}$$ (see Example $\ref{Bez}$), hence $B_p=\{0,1\}^{\mathbb{N}}_p$ can be viewed as Hamming subspace of the Besicovitch space $(B,r_b)$, $$B_p=\cup_{u\in  \mathbb{SN}} \mathcal{H}(u).$$ The Hamming space $B_p$ is locally standard.
\end{example}

\begin{theorem} \label{TH_1} Let $H$ be a locally standard  countable Hamming space. Then $H\cong \otimes_{i=1}^{\infty} \ H_{p_i},$ each $p_i$ is a prime number.
\end{theorem}
\begin{proof} There exists an ascending chain of subspaces $$1 \in  H_1 \subset H_2 \subset \cdots , \quad \bigcup_{i=1}^{\infty} H_i =H,$$ and each $H_i$ is isomorphic to a standard   Hamming space. By Lemma \ref{lemma_1} there exists a  standard subspace $H_i^{\,'} \subset H_{i+1},$ $i \geq 1,$ such that $H_i \otimes H_i^{\,'} \cong  H_{i+1}.$  Let $H_1^{\,'}= H_1.$ Then $H \cong \otimes_{i=1}^{\infty} H_i^{\,'} .$ Suppose now that $H_i^{\,'} \cong H_{n_i}.$ If $n_i\geq 2$ and $n_i=p_1 \cdots p_k$ is a prime decomposition of $n_i$ then $H_{n_i} \cong H_{p_1} \otimes \cdots \otimes H_{p_k}.$ This implies the assertion of the Theorem. \end{proof}


\begin{definition}	\label{Steinitz numbers}  Let $H$ be a locally standard Hamming space. Let $$ D(H) = \{ n\geq 1 \, | \, H^{\,'} \subset H, H^{\,'} \cong H_n  \}.$$ The least common multiple of the set $ D(H)$ is called the Steinitz  number  of  $H$ and denoted as $\mathbf{st}(H).$
\end{definition}

Let $ H^{\,'},$  $H^{\,''}$ be   locally standard Hamming spaces. It is easy to see that $H^{\,'} \otimes H^{\,''}$ is  locally standard and $$\mathbf{st}( H^{\,'} \otimes H^{\,''} ) =\mathbf{st}( H^{\,'}) \cdot \mathbf{st}( H^{\,''}) . $$  If $H= \otimes_{i=1}^{\infty} H_{p_i}$ is a decomposition of Theorem  \ref{TH_1} then $$\mathbf{st}( H)= \prod_{i=1}^{\infty} p^{s_i}_i , $$ where $s_i$ is a number of copies of $H_{p_i}$ in the decomposition of $H$.   In \cite{Sushch2} it was shown that countable locally standard Hamming space are isomorphic if and only if $\mathbf{st}( H^{\,'}) =\mathbf{st}( H^{\,''}) .$ This fact also easily follows from Theorem   \ref{TH_1}.

\begin{example}
The Steinitz number of the $u$-periodic Hamming space $\mathcal{H}(u)$ is equal to $u$. Moreover, if  	$u = \prod_{i=1}^{\infty} p_i^{s_i}$ then $$\mathcal{H}(u) = \otimes_{i=1}^{\infty} H_{p_i},$$
where  the number of copies of $H_{p_i}$ in the tensor products $\otimes_{i=1}^{\infty} H_{p_i}$ is equal to $s_i$.
\end{example}

\begin{example}
	$\mathbf{st}( B_p) =\prod_{i=1}^{\infty} p_i^{\infty}$, where the product is taken over all prime numbers $p_i$.
\end{example}
	
\section{Cartan subalgebras of locally matrix algebras}

Let $\mathbb{F}$ be an algebraically closed field. An associative  $\mathbb{F}$-algebra $A$ with a unit $1$ is said to be a  {\it unital locally matrix algebra} (see \cite{Kurosh}) if for an arbitrary finite collection of elements $a_1, \ldots, a_s \in A$ there exists a subalgebra $A' \subset A$ such that $1,a_1, \ldots, a_s \in A'$ and $A'  \cong M_n(\mathbb{F})$ for some $n \geq 1$.

For a unital locally matrix algebra $A$ let $D(A)$ be the set of all positive integers $n$ such that there exists a subalgebra $A'$, $1 \in A' \subset A$, $A' \cong M_n(\mathbb{F})$. The least common multiple of the set $D(A)$ is called the {\it Steinitz number} $\mathbf{st}(A)$ of the algebra $A$ (see \cite{BezOl}).

G.~K\"{o}the \cite{Koethe} showed that every countable dimensional unital locally matrix algebra is isomorphic to an infinite tensor product of finite dimensional matrix algebras.

J.~G.~Glimm \cite{Glimm} proved that every countable dimensional unital locally matrix algebra is uniquely determined by its Steinitz number.

For unital locally matrix algebras of uncountable dimensions the theorems above are no longer  true (see \cite{BezOl}, \cite{BezOl_2}, \cite{Kurosh}). 

{\it In what follows we consider only countable dimensional locally matrix algebras}.

For an element $a \in A$ choose a subalgebra $A'\subset A$ such that $1,a \in A'$, $A'  \cong M_n(\mathbb{F})$. Let $r_{A'}(a)$ be the rang of the matrix $a$ in $ M_n(\mathbb{F})$. As shown by Kurochkin \cite{Kurochkin} (see also \cite{Morita}) the ratio $r(a)=\frac{1}{n}r_{A'}(a)$ does not depend on the choice of the subalgebra $A'$. We call $r(a)$ the {\it relative rang} of the element $a$. Clearly, $r(a)=0$ if and only if $a=0$. If $a$ is an idempotent (we call $0$ and $1$ idempotents as well) then $r(a)=1$ if and only if $a=1$. Moreover, if $a$ and $b$ are orthogonal idempotents then $r(a+b)=r(a)+ r(b)$.

Let $C$ be a commutative subalgebra of a locally matrix algebra $A$, $1 \in C$. Let $E(C)$ be the set of all idempotents from $C$ (including $0$ and $1$). For idempotents $e,f \in E(C)$   let $ef$, $e+f-2ef$ be their Boolean product and Boolean sum respectively. The Boolean algebra $ E(C)$ with the  relative rang function $r:\ E(C) \to [0,1]$ make $ E(C)$ a Hamming space.

A subalgebra $H$ of the matrix algebra $ M_n(\mathbb{F})$ is called a {\it Cartan subalgebra} if $H \cong \underbrace{\mathbb{F} \oplus \ldots \oplus \mathbb{F}}_{n}$, in other words, $H$ is spanned by $n$ pairwise orthogonal idempotents. It is well known that every Cartan  subalgebra is conjugate of the diagonal subalgebra of  $ M_n(\mathbb{F})$.


Let $1\in A_1 \subset A_2 \subset \ldots$ be an ascending chain of matrix subalgebras such that $A=\cup_{i=1}^{\infty}A_i$. In each $A_i$ choose a Cartan subalgebra $H_i$ so that $1 \in H_1 \subset H_2 \subset \ldots$.  We call
$$H= \bigcup_{i=1}^{\infty}H_i$$ a {\it general Cartan subalgebra} of $A$. As above, $r: \ A \to [0,1]$ is a relative rang function. Then $( E(H),r)$ is a locally standard Hamming space.

A subalgebra $H\subset A$ is called a {\it Cartan subalgebra} if there exists a decomposition $A= \otimes_{i=1}^{\infty}A_i$ into a product of finite dimensional matrix algebras and Cartan subalgebras $H_i$ in $A_i$ such that $H= \otimes_{i=1}^{\infty}H_i$.

\begin{theorem}
	\label{Thm2}
	Any two Cartan subalgebras of $A$ are conjugate via an automorphism of $A$.
\end{theorem}

\begin{proof}
		Let $H'$, $H^{''}$ be Cartan subalgebras corresponding to tensor decompositions $$A\cong\otimes_{i=1}^{\infty}M_{n_i}(\mathbb{F}), \ A\cong\otimes_{i=1}^{\infty}M_{m_i}(\mathbb{F}), \ H'= \otimes_{i=1}^{\infty} H'_i, \ H^{''}= \otimes_{i=1}^{\infty}H^{''}_i,$$ where $H'_i$, $H^{''}_i$ are Cartan subalgebras in $M_{n_i}(\mathbb{F}),$ $M_{m_i}(\mathbb{F})$ respectively. Without loss of generality we can assume that all integers $n_i$, $m_i$ are prime. From $$\mathbf{st}(A)=\prod_{i=1}^{\infty} n_i=\prod_{i=1}^{\infty} m_i$$ it follows that up to renumeration   we can assume $n_i=m_i$. There exist automorphisms $\varphi _i \in \text{Aut } M_{n_i}(\mathbb{F}) $ such that
	$\varphi _i( H'_i)=H^{''}_i$. Now  $H'$ and  $H^{''}$ are conjugate via the automorphism $\varphi= \otimes_{i=1}^{\infty}\varphi _i$.
\end{proof}

It easy to see that a Cartan subalgebra is a general Cartan subalgebra. The reverse statement is not true. In particular, not all general Cartan subalgebras are conjugate.

\begin{theorem}
	\label{Thm3}
	In an arbitrary countable dimensional locally matrix algebra there exists a general Cartan subalgebra that is not a Cartan subalgebra.
\end{theorem}

Let $A^{*}$ denote the group of invertible elements of the algebra $A$.

\begin{lemma} \label{Lem2}
	If $H\subset A$ is a Cartan subalgebra then there exists an element $x \in A^{*} \setminus H$ such that $x^{-1}Hx=H$.
\end{lemma}
\begin{proof}
	Let $A=A_1 \otimes A_2$, $A_1 \cong M_{n}(\mathbb{F}) $, $n \ge 2$;
	$A_2$ is a locally matrix algebra. Let $H_1$, $H_2$ be Cartan subalgebras of the algebras $A_1$, $A_2$ respectively, $H=H_1 \otimes H_2$. There exists an element $a \in A_1^{*} \setminus H_1 $ such that $a^{-1}H_1 a=H_1$. Now it remains to choose $x=a \otimes 1$.
\end{proof}

\begin{proof}[Proof of Theorem \ref{Thm3}]
	Let $A$ be a countable dimensional locally matrix algebra. In view of Lemma \ref{Lem2} it is sufficient to find a general Cartan subalgebra $H$ such that for an arbitrary invertible element $x\in A^{*}$ either $x\in H$ or $x^{-1} H x \neq H.$
	
	Choose an ascending  chain of matrix subalgebras of algebra $A$ such that $$ 1\in A_1 \subset A_2 \subset \cdots , \ \cup_{i=1}^{\infty} A_i = A, \ A_i \cong M_{n_i} (\mathbb{F}), \  n_i\ge 2.  $$ In the case of a finite field $\mathbb{F}$ we assume also that $n_i^2 \leq n_{i+1}$. We will use induction to construct an ascending   chain of Cartan subalgebras $H_k \subset A_k,$ $ H_k \subset H_{k+1},$ $k\ge 1. $
	
	The Cartan subalgebra $H_1 \subset A_1$ is selected arbitrary. Suppose that Cartan subalgebras $H_1 \subset H_2 \subset \cdots \subset H_k$ have been selected.
	
	The algebra $H_k$ is isomorphic to a direct sum of $n_k$ copies of the field $\mathbb{F}.$ Let $e_1,$ $\ldots,$ $e_{n_k}$ be pairwise orthogonal idempotents, $H_k = \mathbb{F}e_1 \oplus \cdots \oplus \mathbb{F} e_{n_k}.$
	
	Let $A_k^{\,'}$ be the centraliser of the subalgebra $A_k$ in $A_{k+1} .$ Then $$A_k^{\,'} \cong M_{n_{k+1}/{n_k}}(\mathbb{F}) \text{ and } A_{k+1}\cong A_k \otimes_{\mathbb{F}} A_k^{\,'}.$$  If the field $\mathbb{F}$ is infinite then the algebra $A_k^{\,'}$ contains infinitely many distinct Cartan subalgebras. If the field $\mathbb{F}$ is finite then $A_k^{\,'}$ contains at least $n_k$ distinct Cartan subalgebras. In any case we choose distinct Cartan subalgebras $H_1^{\,'},$ $\ldots,$ $H_{n_k}^{\,'}$ of $A_k^{\,'}.$  	Let $$H_{k+1}= e_1\otimes H_1^{\,'} + \cdots + e_{n_k}\otimes H_{n_k}^{\,'}.$$  It is easy to see that $H_{k+1}$ is a Cartan subalgebra of $A_{k+1}.$ Since every $H_i^{\,'}$ contains the identity element of $A_k^{\,'}$ it follows that $H_k \subset H_{k+1}.$ It is also easy to see that  $H_{k+1}\cap A_k = H_k.$
	
	The union $H=\cup^{\infty}_{k=1} H_k$ is a general Cartan subalgebra of $A.$ For an arbitrary $k\geq1$ we have $H\cap A_k=H_k.$
	
	Let $x\in A^{*}$ be an invertible element. There exists $k\geq 1$ such that $x\in A_k.$ If $x\in H_k$ or $x^{-1}H_k x\neq H_k$ then we are done.
	
	Suppose that $x\in A_k \setminus H_k$ and $x^{-1}H_k x=H_k.$ Then $x^{-1}e_i x = e_{\pi(i)},$ where $\pi$ is a permutation on $1,2,\ldots,n_k.$ If $\pi = 1$ then $x$ lies in the centralizer of $H_k,$ hence $x\in H_k,$ which contradicts our assumption.
	Therefore $\pi \neq 1.$ Now $$x^{-1}H_{k+1}x=\Sigma_{i=1}^{n_k}e_{\pi(i)}\otimes H_i^{'}\neq H_{k+1}$$ since all Cartan subalgebras $H_1^{'},\ldots,H_{n_k}^{'}$ are distinct.
	This implies $x^{-1}H x\neq H$ and completes the proof of the Theorem.
\end{proof}
	
	\begin{theorem}
		\label{Thm4}
		\begin{enumerate}
			\item[$(1)$] An arbitrary countable locally standard Hamming space $S$ is isomorphic to $E(H) ,$ where $H$ is a  Cartan subalgebra of a countable dimensional locally matrix algebra $A,$ $\mathbf{st}(S)=\mathbf{st}(A).$
			\item[$(2)$] Let $A_1,$ $A_2$ be locally matrix algebras with Cartan subalgebras $H_1,$ $H_2$ respectively. The Hamming spaces $E(H_1),$  $E(H_2)$ are isomorphic if and only if $A_1 \cong A_2.$
		\end{enumerate}	
	\end{theorem}
	\begin{proof} By Theorem \ref{TH_1} $S\cong \otimes_{i=1}^{\infty} H_{p_i},$ where each $p_i$ is a prime number. Consider the corresponding matrix algebras $M_{p_i}(\mathbb{F})$ and their diagonal Cartan subalgebras $D_i.$ Let $$A=\otimes_{i=1}^{\infty} M_{p_i}(\mathbb{F}).$$ Then $D= \otimes_{i=1}^{\infty} D_{i}$ is a Cartan subalgebra of $A.$ We have $$E(D)=\otimes_{i=1}^{\infty} E(D_i) \cong \otimes_{i=1}^{\infty} H_{p_i} \cong S.$$ This proves the part (1).
		
		If $E(H_1)\cong E(H_2)$ then $\mathbf{st}(A_1)=\mathbf{st}(E(H_1))=\mathbf{st}(E(H_2))=\mathbf{st}(A_2).$ By Glimm's Theorem \cite{Glimm} $A_1\cong A_2.$ This completes the proof of the Theorem.
	\end{proof}

\end{document}